\documentclass[11pt]{amsart}
\usepackage{amssymb, latexsym}
\theoremstyle{plain}
\newtheorem{theorem}{Theorem}
\newtheorem{corollary}{Corollary}

\newtheorem {lemma}{Lemma}
\newtheorem{proposition}{Proposition}

\newcommand{\Aa}{A^{U^2D}}
\newcommand{\Bb}{B^{U^2D}}
\newcommand{\A}{\mathcal{A}^{U^2D}}
\newcommand{\B}{\mathcal{B}^{U^2D}}

\numberwithin{equation}{section}

\begin{document}

\title[Longest subsequence of repeated up/down patterns  ] {Longest subsequence for certain repeated up/down patterns in random permutations avoiding a pattern of length three}

\author{Ross G. Pinsky}


\address{Department of Mathematics\\
Technion---Israel Institute of Technology\\
Haifa, 32000\\ Israel}
\email{ pinsky@technion.ac.il}

\urladdr{https://pinsky.net.technion.ac.il/}

\subjclass[2010]{60C05, 05A05} \keywords{}
\date{}

\begin{abstract}

Let $S_n$ denote the set of permutations of $[n]$ and let $\sigma=\sigma_1\cdots\sigma_n\in S_n$. For
any subsequence $\{\sigma_{i_j}\}_{j=1}^k$ of $\{\sigma_i\}_{i=1}^n$ of length $k\ge2$, construct
 the ``up/down'' sequence $V_1\cdots V_{k-1}$ defined by
$$
V_j=\begin{cases} U,\ \text{if}\ \sigma_{i_{j+1}}-\sigma_{i_j}>0;\\ D,\ \text{if}\ \sigma_{i_{j+1}}-\sigma_{i_j}<0,\end{cases}
$$
where  $U$ refers to ``up'',  $D$  to  ``down'' and $V$ to  ``vertical''.
Consider now
a fixed up/down pattern:  $V_1\cdots V_l$, where $l\in\mathbb{N}$ and $V_j\in\{U, D\},\ j\in[l]$. Given a permutation $\sigma\in S_n$,
consider  the length of the longest  subsequence of $\sigma$ that repeats this pattern. Incomplete patterns are not counted, so the  length is necessarily either 0 or of the form
$kl+1$, where $k\in \mathbb{N}$.
 For example, consider  $l=3$ and $V_1V_2V_3=UUD$.
Then for the   permutation $342617985\in S_9$, the length of the  longest subsequence  that repeats the pattern $UUD$ is 7; it is obtained by three different subsequences, namely
3461798, 3461795 and 3461785.

The above framework includes one  much studied case as well as another case that has been studied to some degree.
The pattern $U$  is the celebrated case of the  longest increasing subsequence.
The pattern $UD$ (or $DU$) is the case of   the longest alternating subsequence.
These have been studied both under the uniform distribution on $S_n$ as well as under the uniform distribution on those permutations in $S_n$ which avoid a particular pattern of length three.

In this paper, we consider the patterns $UUD$ and $UUUD$ under the uniform distribution on those permutations in $S_n$ that avoid the pattern $132$.
We prove that the expected value of the longest increasing subsequence  following the pattern $UUD$ is asymptotic to $\frac37n$ and the expected value of the longest increasing subsequence  following the pattern $UUUD$ is asymptotic to $\frac4{11}n$. (For $UD$ (alternating subsequences) it is known to be $\frac12n$.)
This leads directly to appropriate  corresponding results for permutations  avoiding any particular pattern of length three.
\end{abstract}

\maketitle
\section{Introduction and Statement of Results}\label{intro}

Let $S_n$ denote the set of permutations of $[n]=\{1,\cdots, n\}$ and let $\sigma=\sigma_1\cdots\sigma_n\in S_n$. For
any subsequence $\{\sigma_{i_j}\}_{j=1}^k$ of $\{\sigma_i\}_{i=1}^n$ of length $k\ge2$, construct
 the ``up/down'' sequence $V_1\cdots V_{k-1}$ defined by
$$
V_j=\begin{cases} U,\ \text{if}\ \sigma_{i_{j+1}}-\sigma_{i_j}>0;\\ D,\ \text{if}\ \sigma_{i_{j+1}}-\sigma_{i_j}<0,\end{cases}
$$
where  $U$ refers to ``up'', $D$ refers to ``down'' and $V$ refers to ``vertical''.

Consider now
a fixed up/down pattern:  $V_1\cdots V_l$, where $l\in\mathbb{N}$ and $V_j\in\{U, D\},\ j\in[l]$. Given a permutation $\sigma\in S_n$,
consider  the length of the longest  subsequence of $\sigma$ that repeats this pattern. Incomplete patterns are not counted, so the  length is necessarily either 0 or of the form
$kl+1$, where $k\in \mathbb{N}$.
 For example, consider  $l=3$ and $V_1V_2V_3=UUD$.
Then for the   permutation $342617985\in S_9$, the length of the  longest subsequence  that repeats the pattern $UUD$ is 7; it is obtained by three different subsequences, namely
3461798, 3461795 and 3461785. On the other hand, for the permutation 319652478, the length of the longest
subsequence  that repeats the pattern $UUD$ is 0 because this pattern  does not appear at all.

The above framework includes one very celebrated and much studied case as well as another case that has been studied to some degree.
The pattern $U$ is the case of the longest increasing subsequence.
This celebrated case was studied by Logan and Shepp \cite{LS} and Vershik and Kerov \cite{VK}. Their worked showed that the expected value of the length of the longest increasing subsequence in
a uniformly random permutation from $S_n$ behaves asymptotically as $2\sqrt n$. More precise information on the behavior of this random variable was obtained later
in the seminal  paper of Baik, Deift and Johansson \cite{BDJ}; for more on the longest increasing subsequence and many additional references, see the book by Romik \cite{R}.

The pattern $UD$ (or $DU$) is the case of the  longest alternating subsequence (in the first case starting with increasing and ending with decreasing, and in the second case vice versa).
Stanley \cite{S} investigated alternating sequences and showed that the expected value of the longest alternating  subsequence in a uniformly random permutation from $S_n$ behaves asymptotically as $\frac23 n$.
(Of course this asymptotic behavior is independent of how we define the initial or terminal direction in the sequence.) See also further results  by Widom \cite{W}.
The analysis in the alternating case is simpler than in the increasing case because, as  Stanley noted,  there is always a longest alternating subsequence (either beginning with down or ending with up)  of $\sigma\in S_n$ which contains the number $n$. Thus, a longest subsequence can be broken down into smaller pieces that are concatenated.

In \cite{S}, Stanley also posed the question of  whether it is true that  for any pattern of ups and downs as we have defined above, there  exist constants $\mu$ and $c$ such that
the expected value of the length of the longest subsequence repeating this pattern in a uniformly random permutation from $S_n$ behaves asymptotically as $\mu n^c$.
The recent paper \cite{ABLLP}  answered this question in the affirmative, and showed in particular that  for every pattern except for the pattern $U$ corresponding to the longest increasing subsequence, one has
$c=1$. Thus, for every pattern except for $U$, the asymptotic behavior of the expected value of the  length of the longest subsequence repeating that pattern is $\mu n$, for some $\mu\in(0,1)$.
The authors of \cite{ABLLP}  did not explicitly calculate
the value of $\mu$; however, they constructed a dynamical system that can be  used to approximate $\mu$.
They also proved a central limit theorem for the length of the longest subsequence repeating any particular pattern, except for the pattern $U$.

In this paper, instead of considering a permutation that is  uniformly random from $S_n$, we consider
 a permutation that is uniformly random from those permutations in $S_n$ which avoid a certain
 pattern of length three.
Our aim is to calculate explicitly the  asymptotic behavior of the expected value of the length of the longest subsequence repeating certain patterns of ups and downs in such a random permutation.
Before continuing to explain this,
we recall the definition of pattern-avoiding for permutations.
If $\sigma=\sigma_1\cdots\sigma_n\in S_n$ and $\tau=\tau_1\cdots\tau_m\in S_m$, where $2\le m<n$,
then we say that $\sigma$ contains $\tau$ as a pattern if there exists a subsequence $1\le i_1<i_2<\cdots<i_m\le n$ such
that for all $1\le j,k\le m$, the inequality $\sigma_{i_j}<\sigma_{i_k}$ holds if and only if the inequality $\tau_j<\tau_k$ holds.
If $\sigma$ does not contain $\tau$, then we say that $\sigma$  avoids $\tau$.
We consider here permutations on $S_n$ that avoid a pattern $\tau\in S_3$.
Denote by $S_n(\tau)$  the set of permutation in $S_n$ that avoid $\tau$. It is well-known  that
$|S_n(\tau)|=C_n$,  for all six permutations $\tau\in S_3$, where $C_n=\frac1{n+1}\binom{2n}n$, $n\in\mathbb{N}$, is the $n$th
Catalan number \cite{B}.
Let $P_n^{\text{av}(\tau)}$ denote the uniform  probability measure on $S_n(\tau)$ and let $E_n^{\text{av}(\tau)}$ denote the corresponding expectation.

As already noted, the pattern $U$ corresponds to increasing subsequences.
In \cite{DHW}, the asymptotic behavior of the expectation   of the longest increasing subsequence $L_n$ of a random permutation under the distribution
$P_n^{\text{av}(\tau)}$ was obtained for all six permutations $\tau\in S_3$. Of course, the case $\tau=123$ is trivial. The
 expectation is on the order $n$ only for $\tau\in\{231,312,321\}$.
The asymptotic behavior of the variance $v_n(\tau)$ was also investigated, and
the limiting distribution
of $\frac{L_n-E_n^{\text{av}(\tau)}L_n}{v_n(\tau)}$ was calculated, the limit being
Gaussian only for  $\tau\in\{231,312\}$. Large deviations were considered in \cite{P}.

As already noted,  the pattern $UD$ (or $DU$) corresponds to alternating subsequences,
In \cite{FMW}, the asymptotic behavior of the expectation  of the longest alternating sequence
of a random permutation under the distribution
$P_n^{\text{av}(\tau)}$ was shown to be $\frac n2$ for all six choices of $\tau\in S_n$.
The asymptotic variance was also obtained as well as a central limit theorem. Large deviations were considered in \cite{P}.

In this paper, for the patterns $UUD$ and $UUUD$, we will calculate the asymptotic behavior of the expectation of the longest subsequence repeating that pattern in a uniformly random permutation avoiding the pattern
$132$. The proof in the case of $UUD$ involves analyzing the asymptotic behavior of the  coefficients of either of  two generating functions that satisfy  a system of two linear equations. The calculations are somewhat involved. The proof in the case $UUUD$ involves analyzing the asymptotic behavior of the coefficients of any one of three generating functions
that satisfy a system of three linear equations. Here the calculation are quite involved. Using the same method, we could also obtain the asymptotic behavior of the variance, but we have decided not to pursue this, as those calculations would be even more involved.
In principle, our technique can be continued for the pattern $U^lD$, for any $l\in \mathbb{N}$, where $U^l$ indicates $l$ consecutive $U$'s.
However, this involves solving a system of $l$ linear equations for $l$ generating functions, solving explicitly for one of them,  and then analyzing its coefficients. It also involves the solving of an auxiliary set
of equations to calculate the probability that $\sigma\in S_n(132)$ does not have an increasing subsequence of length $j$, for $j=1,\cdots, l$ (an extension of Lemma \ref{probA=0lemma} in Section \ref{sec3}).
The reason the cases $U^lD$ are in principle tractable for 132-avoiding permutations is that a  variant of Stanley's observation holds in these cases; namely, that for a 132-avoiding permutation,  there is always either  a  longest subsequence repeating the pattern $U^lD$ that contains the number $n$, or else, every such longest subsequence starts after the appearance of the number $n$ in the permutation.

Using the reversal, complementation, and reversal-complementation operations for permutations, the  results we obtain for permutations avoiding the pattern 132  can be translated into similar results for permutations avoiding any one of the patterns
$213,231, 312$. Using a well-known bijection between permutations avoiding the pattern 132 and permutations avoiding the pattern 123, along with reversal, the results we obtain
 can be translated into similar results for
permutations avoiding either of the patterns 123, 321.

We now state two theorems for 132-avoiding permutations, one for the pattern $UUD$ and one for the pattern $UUUD$, and then state a corollary of these two theorems
that contains similar results for $\tau$-avoiding permutations for the other five $\tau\in S_3$.
\begin{theorem}\label{thmUUD}
Let $L^{U^2D}_n(\sigma)$ denote the length of the longest subsequence of the repeated pattern $UUD$ in $\sigma\in S_n(132)$. (So $L^{U^2D}_n(\sigma)$ is either equal to 0 or to $3k+1$ for some $k\in\mathbb{N}$.)
Then
\begin{equation}\label{UUDexp}
E_n^{\text{av}(132)}L^{U^2D}_n\sim\frac37n.
\end{equation}
\end{theorem}
\begin{theorem}\label{thmUUUD}
Let $L^{U^3D}_n(\sigma)$ denote the length of the longest subsequence of the repeated pattern $UUUD$ in $\sigma\in S_n(132)$. (So $L^{U^3D}_n(\sigma)$ is either equal to 0 or to $4k+1$ for some $k\in\mathbb{N}$.)
Then
\begin{equation}\label{UUUDexp}
E_n^{\text{av}(132)}L^{U^3D}_n\sim\frac4{11}n.
\end{equation}
\end{theorem}
\bf\noindent Remark.\rm\ Recall that it was noted above  that for the repeated pattern $UD$ (which corresponds to alternating subsequences), the corresponding asymptotic behavior is $\frac12n$.
\begin{corollary}\label{cor}
\noindent i.
Let $L^{V_1V_2V_3}_n(\sigma)$ denote the length of the longest subsequence of the repeated pattern $V_1V_2V_3$ in $\sigma\in S_n$, where $V_i\in \{U,D\}$, $i=1,2,3$.
Then
\begin{equation}\label{corUUDexp}
E_n^{\text{av}(\tau)}L^{V_1V_2V_3}_n\sim\frac37n,
\end{equation}
for the following five pairs of $V_1V_2V_3$ and $\tau$:  $UDD$ and 231;  $DDU$ and 312;  $DUU$ and 213; $UDD$ and 123;  $UUD$ and 321.

\noindent ii. Let $L^{V_1V_2V_3V_4}_n(\sigma)$ denote the length of the longest subsequence of the repeated pattern $V_1V_2V_3V_4$ in $\sigma\in S_n$, where $V_i\in \{U,D\}$, $i=1,2,3,4$.
Then
\begin{equation}\label{corUUUDexp}
E_n^{\text{av}(\tau)}L^{V_1V_2V_3V_4}_n\sim\frac4{11}n,
\end{equation}
for the following five pairs of $V_1V_2V_3$ and $\tau$:  $UDDD$ and 231;  $DDDU$ and 312;  $DUUU$ and 213; $UDDD$ and 123;  $UUUD$ and 321.

\end{corollary}

\noindent \it Proof of Corollary.\rm\
Recall that the \it reverse\rm\  of a permutation $\sigma=\sigma_1\cdots\sigma_n$ is the permutation $\sigma^{\text{rev}}:=\sigma_n\cdots\sigma_1$,
and the \it complement\rm\ of $\sigma$ is the permutation
$\sigma^{\text{com}}$ satisfying
$\sigma^{\text{com}}_i=n+1-\sigma_i,\ i=1,\cdots, n$.
Let $\sigma^{\text{rev-com}}$ denote the permutation obtained by applying reversal and then complementation to $\sigma$ (or equivalently, vice versa).
Since $132^{\text{rev}}=231$, $132^{\text{comp}}=312$ and $132^{\text{rev-com}}=213$, if follows that
the three operations, reversal, complementation and reversal-complementation, are bijections from $S_n(132)$ to $S_n(\tau)$, with $\tau=231$ in the case of reversal, $\tau=312$ in the case of complementation and $\tau=213$ in the case of reversal-complementation.
From these facts and Theorems \ref{thmUUD} and \ref{thmUUUD},
the corollary follows immediately for $\tau\in\{231, 312, 213\}$.

There is a well-known explicit bijection between $S_n(132)$ and $S_n(123)$ \cite{SS,B}. Recall that an entry $j\in[n]$ of a permutation
$\sigma\in S_n$ is called a \it left-to-right minimum \rm if
$\sigma_j=\min\{\sigma_i:1\le i\le j\}$.
For  a permutation $\sigma\in S_n(132)$, let $\{i_j\}_{j=1}^k$ denote its left-to-right minima.  Then necessarily
the entries of $\sigma$ that appear from left to right between $\sigma_{i_j}$ and $\sigma_{i_{j+1}}$ (or after  $\sigma_{i_k}$ up through the final term in the permutation)  are  increasing, with each entry being the smallest number remaining that is larger than its predecessor. (In particular, the left most such entry is the smallest remaining number larger than $\sigma_{i_j}$.) The bijection between
$S_n(132)$ and $S_n(123)$ preserves the set of  left-to-right minima, and then rearranges all of the other entries in descending order from left to right.
Note that the values of the permutation at the left-to-right minima  form a decreasing sequence, and these other rearranged entries also form a decreasing sequence; thus the permutation obtained is the union of two decreasing sequences, which is equivalent to its being 123-avoiding. One can check easily that for the pattern $UUD$ (or $UUUD$), there is either no copy  or  one copy of the pattern
between two consecutive left-to-right minima in the permutation $\sigma\in S_n(132)$, and that the same number of copies of $UDD$ (or $UDDD$) appear between those two consecutive left-to-right minima in the $123$-avoiding permutation obtained from $\sigma$ via the above described bijection. This proves the corollary for $\tau=123$. Applying reversal to 123 proves the corollary for $\tau=321$.
\hfill $\square$

We prove Theorem \ref{thmUUD} in Section \ref{pfthm1}.
We derive a system of  two linear equations for two generating functions, and then solve explicitly for one of them. These generating functions are connected to the expected number of complete
patterns $UUD$ in a maximal subsequence.
The leading order  asymptotic behavior of the coefficients of either of these generating functions is equal to the leading order asymptotic behavior of $\frac13C_nE_n^{\text{av}(132)}L^{U^2D}_n$.
Performing an asymptotic analysis of the coefficients of this generating function yields the proof of the theorem.

The proof of Theorem \ref{thmUUUD} is much longer. In Section \ref{sec3} we derive a system of three linear equations for three generating functions, and then solve explicitly for one of them. The explicit expression for this generating function is quite involved.
 These generating functions are connected to the expected number of complete
patterns $UUUD$ in a maximal subsequence.
The leading order  asymptotic behavior of the coefficients of any of these three generating functions is equal to the leading order asymptotic behavior of $\frac14C_nE_n^{\text{av}(132)}L^{U^3D}_n$.
In Section \ref{sec4} we perform a lot of algebra in order to obtain the generating function in a more manageable form.  Then we perform
an asymptotic analysis of the coefficients of this generating function to yield the proof of the theorem.

\section{Proof of Theorem \ref{thmUUD}}\label{pfthm1}
For $\sigma\in S_n$ and $n\in\mathbb{N}$, define $B^{U^2D}_n(\sigma)$ to be the number of complete sets of $UUD$ in a longest subsequence in $\sigma$ of the repeated pattern $UUD$.
Thus,
\begin{equation}\label{BnLn}
B^{U^2D}_n(\sigma)=\begin{cases}\frac13\left(L_n^{U^2D}(\sigma)-1\right), \ \text{if}\ L_n^{U^2D}(\sigma)\neq0;\\ 0, \text{if}\ L_n^{U^2D}(\sigma)=0.\end{cases}
\end{equation}
Also, for convenience, we define $B^{U^2D}_0\equiv0$.

For $\sigma\in S_n$ and $n\in\mathbb{N}$, define $A^{U^2D}_n(\sigma)=0$, if $\sigma=n\cdots 21$; otherwise,  find a longest   subsequence $\{\sigma_{i_j}\}_{j=1}^{3k+2}$, $k\in\mathbb{Z}^+$, of $\sigma$ for which the up/down pattern is $UUD\cdots UUDU$, and define
$A^{U^2D}_n(\sigma)=k+1$. For convenience, we define $A^{U^2D}_0\equiv0$.

 In the sequel, for any $j\in\mathbb{N}$,  $\Bb_j$ and $\Aa_j$ will always be considered as random variables on the probability space $\left(S_j(132),P_j^{\text{av}(132)}\right)$.
Define
\begin{equation}\label{anbn}
\begin{aligned}
&b_n=E_n^{\text{av}(132)}B^{U^2D}_n;
&a_n=E_n^{\text{av}(132)}A^{U^2D}_n,
\end{aligned}
\end{equation}
where we have suppressed the notation $U^2D$.
Define the generating functions for $\{C_nb_n\}_{n=0}^\infty$ and $\{C_na_n\}_{n=0}^\infty$ by
\begin{equation}\label{genfuncsAB}
\begin{aligned}
&\B(t)=\sum_{n=0}^\infty C_nb_nt^n;\\
&\A(t)=\sum_{n=0}^\infty C_na_nt^n.\\
\end{aligned}
\end{equation}
Also let $C(t)=\sum_{n=0}^\infty C_nt^t$ denote the generating function of the Catalan numbers, where we define $C_0=1$. As is well-known,
\begin{equation}\label{Catgenfunc}
C(t)=\frac{1-\sqrt{1-4t}}{2t}.
\end{equation}

The following definition will be useful. Let $a_1<a_2\cdots <a_m$ be real numbers and let $\rho=\rho_1\cdots\rho_m$ be a permutation of these numbers.
We define $\text{red}(\rho)\in S_m$, the reduction of $\rho$, to be the permutation in $S_m$ that has the same pattern as $\rho$. That is,
$\text{red}(\rho)=\sigma$ if $\sigma$ satisfies $\sigma_i<\sigma_j$ whenever $\rho_i<\rho_j$, $i,j\in[m]$.
Note that the up/down pattern that one can associate with $\rho=\rho_1\cdots\rho_m$ is the same as the up/down pattern associated with $\text{red}(\rho)$.
Every permutation $\sigma\in S_n(132)$  has the property that if $\sigma_j=n$, then the numbers $\{n-j+1,\cdots, n-1\}$ appear in the first $j-1$ positions in $\sigma$ and the numbers
$\{1,\cdots, n-j\}$ appear in the last $n-j$ positions in $\sigma$.
From this fact, along with the fact that $|S_n(132)|=C_n$, it follows that
\begin{equation}\label{nprob}
P_n^{\text{av}(132)}(\sigma_j=n)=\frac{C_{j-1}C_{n-j}}{C_n},\ \text{for}\ j\in[n].
\end{equation}
It also follows that under the conditioned  measure $P_n^{\text{av}(132)}|\{\sigma_j=n\}$, the permutation $\text{red}(\sigma_1\cdots\sigma_{j-1})\in S_{j-1}$ has the distribution
$P_{j-1}^{\text{av}(132)}$,  the permutation $\sigma_{j+1}\cdots\sigma_n\in S_{n-j}$ has the distribution $P_{n-j}^{\text{av}(132)}$, and these two permutations are independent.

We now derive a system  of two linear equations for $\B(t)$ and $\A(t)$, and then solve for one of them explicitly.
From the definitions of $\Bb_n$ and $\Aa_n$, we have
$$
\Bb_n\equiv0,\ 0\le n\le 3;\ \ \Aa_n\equiv0,\ 0\le n\le 1.
$$
Thus,
\begin{equation}\label{ABbegin}
\begin{aligned}
&b_n=0,\ 0\le n\le 3;\\
&a_n=0,\ 0\le n\le 1.\\
\end{aligned}
\end{equation}
The following proposition is the key to obtaining a pair of linear equations for the generating functions $\B(t)$ and $\A(t)$.
\begin{proposition}\label{conddist}
\noindent i.
\begin{equation}\label{Bconddist}
\begin{aligned}
&\Bb_n|\{\sigma_j=n\}\stackrel{\text{dist}}{=} \Aa_{j-1}+\Bb_{n-j},\ j\in[n-1],\ n\ge2;\\
&\Bb_n|\{\sigma_n=n\}\stackrel{\text{dist}}{=}\Bb_{n-1},\ n\ge 2,
\end{aligned}
\end{equation}
where  on the right hand side of \eqref{Bconddist}, $\Aa_{j-1}$ and $\Bb_{n-j}$ are understood to be independent.

\noindent ii.
\begin{equation}\label{Aconddist}
\begin{aligned}
&\Aa_n|\{\sigma_1=n\}\stackrel{\text{dist}}{=}\Aa_{n-1}, \ n\ge2;\\
&\Aa_n|\{\sigma_j=n\}\stackrel{\text{dist}}{=} \left(\Aa_{j-1}+\Aa_{n-j}\right)1_{\{\Aa_{n-j}\neq0\}}+\left(\Bb_{j-1}+1\right)1_{\{\Aa_{n-j}=0\}},\\
& j\in\{2,\cdots, n\},\ n\ge2,
\end{aligned}
\end{equation}
where
 on the right hand side of \eqref{Aconddist},   $\Aa_{j-1}$ and $\Aa_{n-j}$ are understood to be independent and $\Bb_{j-1}$ and $\Aa_{n-j}$  are understood to be independent.
\end{proposition}

\begin{proof}
The first line of \eqref{Bconddist} follows from
the equality
\begin{equation}\label{Bcondsig}
\begin{aligned}
&\Bb_n(\sigma)= \Aa_{j-1}(\text{red}(\sigma_1\cdots\sigma_{j-1}))+\Bb_{n-j}(\sigma_{j+1}\cdots\sigma_n),\ \text{if}\ \sigma_j=n,\\
&\text{for}\ j\in[n-1], n\ge2,
\end{aligned}
\end{equation}
along with the fact noted above that under the conditioned  measure $P_n^{\text{av}(132)}|\{\sigma_j=n\}$, the permutation $\text{red}(\sigma_1\cdots\sigma_{j-1})\in S_{j-1}$ has the distribution
$P_{j-1}^{\text{av}(132)}$,  the permutation $\sigma_{j+1}\cdots\sigma_n\in S_{n-j}$ has the distribution $P_{n-j}^{\text{av}(132)}$, and these two permutations are independent.
Rather than give a formal proof of \eqref{Bcondsig}, we   convince the reader of its validity by giving an  example and then a generic explanation.

Let $\sigma=435768921$. Then $n=9$ and $j=7$.
We have $\Aa_{j-1}\left(\text{red}(\sigma_1\cdots\sigma_{j-1})\right)=\Aa_6\left(\text{red}(435768)\right)=\Aa(213546)=2$, because the subsequence  23546  (as well as 13546)
corresponds to $UUDU$. We have $\Bb_{n-j}(\sigma_{j+1}\cdots\sigma_n)=\Bb_2(21)=0$.
And we have $\Bb_n(\sigma)=\Bb_9(435768921)=2$ because  the subsequence 4576892 (as well as several others) corresponds to $UUDUUD$.

Generically, $\Bb_n(\sigma)$ is the sum of two terms. One of the terms is  $\Aa_{j-1}\left(\text{red}(\sigma_1\cdots\sigma_{j-1})\right)$, which counts  the  number of full sets of $UUD$ and then adds one for an extra  $U$.
This extra $U$, along with $\sigma_j=n$ and $\sigma_{j+1}$
supply an additional full set $UUD$ which was counted by $\Aa_{j-1}\left(\text{red}(\sigma_1\cdots\sigma_{j-1})\right)$ (via the adding one for the extra $U$).  The other term is  $\Bb_{n-j}(\sigma_{j+1}\cdots\sigma_n)$, which counts the remaining sets of $UUD$.

The second line of \eqref{Bconddist} is obtained  using the following rather obvious equality instead of \eqref{Bcondsig}:
$$
\Bb_n(\sigma)=\Bb_{n-1}(\sigma_1\cdots\sigma_{n-1}),\ \text{if}\ \sigma_n=n.
$$

The second line of \eqref{Aconddist}
follows from
the equality
\begin{equation}\label{Acondsig}
\begin{aligned}
&\Aa_n(\sigma)= \left(\Aa_{j-1}(\text{red}(\sigma_1\cdots\sigma_{j-1}))+\Aa_{n-j}(\sigma_{j+1}\cdots\sigma_n)\right)1_{\Aa_{n-j}(\sigma_{j+1}\cdots\sigma_n)\neq0}+\\
&\left(\Bb_{j-1}\left(\text{red}(\sigma_1\cdots\sigma_{j-1})\right)+1\right)1_{\Aa_{n-j}(\sigma_{j+1}\cdots\sigma_n)=0},
\ \text{if}\ \sigma_j=n,\ \text{for}\ j\in[n-1], n\ge2,
\end{aligned}
\end{equation}
along with the fact noted above that under the conditioned  measure $P_n^{\text{av}(132)}|\{\sigma_j=n\}$, the permutation $\text{red}(\sigma_1\cdots\sigma_{j-1})\in S_{j-1}$ has the distribution
$P_{j-1}^{\text{av}(132)}$,  the permutation $\sigma_{j+1}\cdots\sigma_n\in S_{n-j}$ has the distribution $P_{n-j}^{\text{av}(132)}$, and these two permutations are independent.

In the case that
$\Aa_{n-j}(\sigma_{j+1}\cdots\sigma_n)\neq0$, \eqref{Acondsig} is obtained by reasoning similar to that for \eqref{Bcondsig}. We explain \eqref{Acondsig}
in the  case that $\Aa_{n-j}(\sigma_{j+1}\cdots\sigma_n)=0$
with an example. Let $\sigma=435786921$ (slightly different than the $\sigma$ used above). So $n=9$ and $j=7$. We have $\Aa_{n-j}(\sigma_{j+1}\cdots\sigma_n)=\Aa_2(21)=0$.
We have $\Bb_{j-1}(\text{red}(\sigma_1\cdots\sigma_{j-1}))=\Bb_6(\text{red}(435786))=\Bb_6(213564)=1$ because the subsequence 2354 (as well as several  others) corresponds to $UUD$. And
we have $\Aa_n(\sigma)=\Aa_9(435786921)=2$ because  the subsequence 45769 (as well as several others)  corresponds to $UUDU$.
(Note that $\Aa_{j-1}(\text{red}(\sigma_1\cdots\sigma_{j-1}))=\Aa_6(\text{red}(435786))=\Aa_6(213564)=1$ because the subsequence 23 (as well as  several others) corresponds to $U$. Thus,
when $\Aa_{n-j}(\sigma_{j+1}\cdots\sigma_n)=0$,
it is not true in general that
$\Aa_n(\sigma)= \Aa_{j-1}(\text{red}(\sigma_1\cdots\sigma_{j-1}))+\Aa_{n-j}(\sigma_{j+1}\cdots\sigma_n)$.)

The first line of \eqref{Aconddist} is obtained  using the following rather obvious equality instead of \eqref{Acondsig}:
$$
\Aa_n(\sigma)=\Aa_{n-1}(\sigma_2\cdots\sigma_n),\ \text{if}\ \sigma_1=n.
$$

\end{proof}

From \eqref{nprob} and \eqref{Bconddist}, it follows that
\begin{equation}\label{bn}
\begin{aligned}
&b_n=E_n^{\text{av}(132)}\Bb_n=\sum_{j=1}^nE_n^{\text{av}(132)}(B^{U^2D}_n|\sigma_j=n)P_n^{\text{av}(132)}(\sigma_j=n)=\\
&\sum_{j=1}^{n-1}\left(E_{j-1}^{\text{av}(132)}A^{U^2D}_{j-1}+E_{n-j}^{\text{av}(132)}B^{U^2D}_{n-j}\right)\frac{C_{j-1}C_{n-j}}{C_n}+E_{n-1}^{\text{av}(132)}B^{U^2D}_{n-1}
\frac{C_{n-1}C_0}{C_n}=\\
&\sum_{j=1}^{n-1}\left(a_{j-1}+b_{n-j}\right)\frac{C_{j-1}C_{n-j}}{C_n}+b_{n-1}\frac{C_{n-1}C_0}{C_n},\ n\ge2.
\end{aligned}
\end{equation}
Multiplying both sides of  \eqref{bn} by $C_nt^n$, summing over $n$ from 4 to $\infty$, and using \eqref{ABbegin}, we obtain
\begin{equation}\label{Bequ}
\begin{aligned}
&\B(t)=\sum_{n=4}^\infty C_nb_nt^n=t\sum_{n=4}^\infty\left(\sum_{j=1}^{n-1}\left(a_{j-1}+b_{n-j}\right)C_{j-1}C_{n-j}\right)t^{n-1}+\\
&t\sum_{n=4}^\infty b_{n-1}C_{n-1}t^{n-1}.
\end{aligned}
\end{equation}
Straightforward algebraic  calculations along with \eqref{ABbegin} show that
\begin{equation}\label{double1}
\begin{aligned}
&\sum_{n=4}^\infty\left(\sum_{j=1}^{n-1}a_{j-1}C_{j-1}C_{n-j}\right)t^{n-1}=\A(t)\left(C(t)-1\right);\\
&\sum_{n=4}^\infty\left(\sum_{j=1}^{n-1}b_{n-j}C_{j-1}C_{n-j}\right)t^{n-1}=\B(t)C(t).
\end{aligned}
\end{equation}
From \eqref{Bequ} and \eqref{double1}, we obtain
$$
\B(t)=t\left(\A(t)\left(C(t)-1\right)+\B(t)C(t)+\B(t)\right),
$$
which we write as
\begin{equation}\label{finalB}
\B(t)=\frac{t\left(C(t)-1\right)\A(t)}{1-t-tC(t)}.
\end{equation}

Note that for $l\in\mathbb{N}$ and $\sigma\in S_l(132)$,  $\Aa_l(\sigma)=0$ only for  $\sigma=l\cdots 21$;
thus $P_l^{\text{av}(132)}(\Aa_l=0)=\frac1{C_l}$.
Using this with   \eqref{nprob} and \eqref{Aconddist}, it follows that
\begin{equation}\label{an}
\begin{aligned}
&a_n=E_n^{\text{av}(132)}\Aa_n=\sum_{j=1}^nE_n^{\text{av}(132)}(\A_n|\sigma_j=n)P_n^{\text{av}(132)}(\sigma_j=n)=\\
& \frac{C_0C_{n-1}}{C_n}a_{n-1}+  \sum_{j=2}^n\left(a_{j-1}\left(1-\frac1{C_{n-j}}\right)+a_{n-j}\right)\frac{C_{j-1}C_{n-j}}{C_n}+\\
&\sum_{j=2}^n\frac{b_{j-1}+1}{C_{n-j}}\thinspace\frac{C_{j-1}C_{n-j}}{C_n}.
\end{aligned}
\end{equation}
Multiplying both sides of  \eqref{an} by $C_nt^n$, summing over $n$ from 2 to $\infty$ and using \eqref{ABbegin}, we obtain
\begin{equation}\label{aequ}
\begin{aligned}
&\A(t)=\sum_{n=2}^\infty C_na_nt^n=\\
&t\sum_{n=2}^\infty C_{n-1}a_{n-1}t^{n-1}+t\sum_{n=2}^\infty\left(\sum_{j=2}^n a_{j-1}C_{j-1}C_{n-j}\right)t^{n-1}-
t\sum_{n=2}^\infty\big(\sum_{j=2}^n
a_{j-1}C_{j-1}\big)t^{n-1}+\\
&t\sum_{n=2}^\infty\left(\sum_{j=2}^nC_{j-1}a_{n-j}C_{n-j}\right)t^{n-1}+t\sum_{n=2}^\infty\left(\sum_{j=2}^nb_{j-1}C_{j-1}\right)t^{n-1}+t\sum_{n=2}^\infty\left(\sum_{j=2}^n
C_{j-1}\right)t^{n-1}.
\end{aligned}
\end{equation}
Straightforward algebraic calculations along with \eqref{ABbegin} show that
\begin{equation}\label{double2}
\begin{aligned}
&\sum_{n=2}^\infty\left(\sum_{j=2}^n a_{j-1}C_{j-1}C_{n-j}\right)t^{n-1}=\A(t)C(t);\\
&\sum_{n=2}^\infty\left(\sum_{j=2}^nC_{j-1}a_{n-j}C_{n-j}\right)t^{n-1}=\A(t)\left(C(t)-1\right);\\
&\sum_{n=2}^\infty\big(\sum_{j=2}^na_{j-1}C_{j-1}\big)t^{n-1}=\frac{\A(t)}{1-t};\\
&\sum_{n=2}^\infty\big(\sum_{j=2}^nb_{j-1}C_{j-1}\big)t^{n-1}=\frac{\B(t)}{1-t};\\
&\sum_{n=2}^\infty\big(\sum_{j=2}^nC_{j-1}\big)t^{n-1}=\frac{C(t)-1}{1-t}.
\end{aligned}
\end{equation}
From \eqref{aequ} and \eqref{double2}, we obtain
\begin{equation}\label{prefinalA}
\begin{aligned}
&\A(t)=t\Bigg(\A(t)+\A(t)C(t)-\frac{\A(t)}{1-t}+\A(t)\left(C(t)-1\right)+\\
&\frac{\B(t)}{1-t}+\frac{C(t)-1}{1-t}\Bigg),
\end{aligned}
\end{equation}
which we write as
\begin{equation}\label{finalA}
\A(t)=\frac{t\left(\B(t)+C(t)-1\right)}{(1-t)\left(1-2tC(t)\right)+t}.
\end{equation}

Substituting \eqref{finalA} in \eqref{finalB} and solving for $\B(t)$, we obtain
\begin{equation}\label{Balone}
\B(t)=\frac{t^2\left(C(t)-1\right)^2}{\left((1-t)(1-2tC(t))+t\right)\left(1-t-tC(t)\right)-t^2(C(t)-1)}.
\end{equation}
We write the  denominator in \eqref{Balone}
as
\begin{equation}\label{denomB}
\begin{aligned}
&\left((1-t)(1-2tC(t))+t\right)\left(1-t-tC(t)\right)-t^2(C(t)-1)=\\
&t^2-t+1+(-2t^3+3t^2-3t)C(t)+2t^2(1-t)C^2(t).
\end{aligned}
\end{equation}
Using \eqref{Catgenfunc} and performing some algebra, we have
\begin{equation}\label{denomBagain}
\begin{aligned}
&t^2-t+1+(-2t^3+3t^2-3t)C(t)+2t^2(1-t)C^2(t)=\\
&\frac12\left((2t^2-t+1)\sqrt{1-4t}+(1-t)(1-4t)\right).
\end{aligned}
\end{equation}
Using \eqref{Catgenfunc}, the
 numerator in \eqref{Balone} can be written as.
\begin{equation}\label{numer}
\begin{aligned}
t^2(C(t)-1)^2=\frac12\left(2t^2-4t+1+(2t-1)\sqrt{1-4t}\right).
\end{aligned}
\end{equation}
From \eqref{Balone}-\eqref{numer}, we obtain
\begin{equation}\label{Baloneagain}
\B(t)=\frac{2t^2-4t+1+(2t-1)\sqrt{1-4t}}{(2t^2-t+1)\sqrt{1-4t}+(1-t)(1-4t)}.
\end{equation}
In order to eliminate the square root from the denominator, we
multiply the numerator and denominator on the right hand side of \eqref{Baloneagain} by the conjugate of the denominator,
$-(2t^2-t+1)\sqrt{1-4t}+(1-t)(1-4t)$.
Performing the algebra, the new denominator can be written as  $-4t(1-4t)(t^3-t+1)$.
Performing the algebra to calculate the new numerator, and dividing the new numerator and the new denominator by $t(1-4t)$, we obtain
\begin{equation}\label{BUD}
\B(t)=\frac{N(t)}{D(t)},
\end{equation}
where
\begin{equation}\label{U}
\begin{aligned}
&tN(t)=(2t^2-4t+1)(1-t)+(2t-1)(1-t)(1-4t)^\frac12\\
&-(2t^2-4t+1)(2t^2-t+1)(1-4t)^{-\frac12}-(2t-1)(2t^2-t+1)
\end{aligned}
\end{equation}
and
\begin{equation}\label{D}
D(t)=-4(t^3-t+1).
\end{equation}

For a function $f$ represented by a power series as $f(t)=\sum_{n=0}^\infty f_nt^n$, we use the notation $[t^n]f(t)=f_n$.
From \cite[p. 381]{FS}, we have
\begin{equation}\label{1-tpower}
[t^n]\left(1-4t\right)^{-\alpha}=
4^n\frac{n^{\alpha-1}}{\Gamma(\alpha)}\left(1+O\left(\frac1n\right)\right), \ \text{for}\
 \alpha\in \mathbb{C}-\mathbb{Z}_{\le0}.
\end{equation}
From the transfer theorem \cite[Theorem VI.3, p. 390, Example VI.2, p. 395]{FS} and \eqref{1-tpower}, it follows that
if  $g(t)$ is analytic in  a disk, centered at the origin, of radius larger than $\frac14$,
 then
\begin{equation}\label{transfer}
[t^n]g(t)\left(1-4t\right)^{-\alpha}=g(\frac14)4^n\frac{n^{\alpha-1}}{\Gamma(\alpha)}\left(1+O\left(\frac1n\right)\right), \ \text{for}\
 \alpha\in \mathbb{C}-\mathbb{Z}_{\le0}.
\end{equation}
From \eqref{transfer}
it follows immediately that
\begin{equation}\label{transferpower}
[t^n]g(t)t^b\left(1-4t\right)^{-\alpha}=g(\frac14)4^{n-b}\frac{n^{\alpha-1}}{\Gamma(\alpha)}\left(1+O\left(\frac1n\right)\right), \ \text{for}\ b\in\mathbb{Z}\ \text{and} \
 \alpha\in \mathbb{C}-\mathbb{Z}_{\le0}.
\end{equation}

Noting that all of the roots of $D(t)$  have absolute value greater than $\frac14$, and applying  \eqref{transferpower} with $\alpha\in\{\frac12,-\frac12\}$ and $g(t)=-\frac1{D(t)}$ to \eqref{BUD}-\eqref{D},
it follows that the leading order contribution to $[t^n]\B(t)$ as $n\to\infty$ comes from the term
$\frac{\frac1t(2t^2-4t+1)(2t^2-t+1)(1-4t)^{-\frac12}}{4(t^3-t+1)}=\frac{4t^3-10t^2+8t-5+t^{-1}}{4(t^3-t+1)}(1-4t)^{-\frac12}$.  Noting that
 $\Gamma(\frac12)=\sqrt\pi$ and $g(\frac14)=-\frac1{D(\frac14)}=\frac{16}{49}$
it follows from \eqref{transferpower} that
\begin{equation}\label{finalasympb}
[t^n]\B(t)\sim\frac{16}{49}4^n\frac{n^{-\frac12}}{\sqrt\pi}\left(4\cdot4^{-3}-10\cdot4^{-2}+8\cdot4^{-1}-5+4\right)=\frac174^n\frac{n^{-\frac12}}{\sqrt\pi}.
\end{equation}
From \eqref{genfuncsAB} and \eqref{anbn}, we have
$[t^n]\B(t)=C_nb_n=C_nE_n^{\text{av}(132)}B^{U^2D}_n$, and as is well known,
the Catalan numbers satisfy
$C_n\sim4^n\frac{n^{-\frac32}}{\sqrt\pi}$. Using these facts with \eqref{finalasympb}, we conclude that
\begin{equation}\label{finalexpB}
E_n^{\text{av}(132)}B^{U^2D}_n\sim \frac17n.
\end{equation}
Now Theorem \ref{thmUUD} follows from \eqref{finalexpB} and \eqref{BnLn}.
\hfill $\square$

\section{ Derivation of the generating functions for   Theorem \ref{thmUUUD}}\label{sec3}

For $\sigma\in S_n$ and $n\in\mathbb{N}$, define $B^{U^3D}_n(\sigma)$ to be the number of complete sets of $UUUD$ in a longest subsequence in $\sigma$ of the repeated pattern $UUUD$.
Thus,
\begin{equation}\label{BnLn3}
B^{U^3D}_n(\sigma)=\begin{cases}\frac14\left(L_n^{U^3D}(\sigma)-1\right), \ \text{if}\ L_n^{U^3D}(\sigma)\neq0;\\ 0, \text{if}\ L_n^{U^2D}(\sigma)=0.\end{cases}
\end{equation}
Also, for convenience, we define $B^{U^3D}_0\equiv0$.

For $\sigma\in S_n$ and $n\in\mathbb{N}$, define $G^{U^3D}_n(\sigma)=0$, if $\sigma=n\cdots 21$; otherwise,  find a longest   subsequence $\{\sigma_{i_j}\}_{j=1}^{4k+2}$, $k\in\mathbb{Z}^+$, of $\sigma$ for which the up/down pattern is $UUUD\cdots UUUDU$, and define
$G^{U^3D}_n(\sigma)=k+1$. For convenience, we define $G^{U^3D}_0\equiv0$.

For $\sigma\in S_n$ and $n\in\mathbb{N}$, define $A^{U^3D}_n(\sigma)=0$, if $\sigma$ has no increasing subsequence of length three (or equivalently, if $\sigma$ has no subsequence $\{\sigma_{i_j}\}_{j=1}^3$ which corresponds to the pattern $UU$); otherwise,  find a longest   subsequence $\{\sigma_{i_j}\}_{j=1}^{4k+3}$, $k\in\mathbb{Z}^+$, of $\sigma$ for which the up/down pattern is $UUUD\cdots UUUDUU$, and define
$A^{U^3D}_n(\sigma)=k+1$. For convenience, we define $A^{U^3D}_0\equiv0$.

 In the sequel, for $j\in\mathbb{N}$,  $B^{U^3D}_j, G^{U^3D}_j$  and $A^{U^3D}_j$ will always be considered as random variables on the probability space $\left(S_j(132),P_j^{\text{av}(132)}\right)$.
Define
\begin{equation}\label{anbngn}
\begin{aligned}
&b_n=E_n^{\text{av}(132)}B^{U^3D}_n;\
g_n=E_n^{\text{av}(132)}G^{U^3D}_n;\
a_n=E_n^{\text{av}(132)}A^{U^3D}_n,\
\end{aligned}
\end{equation}
where we have suppressed the notation $U^3D$.

Define the generating functions for $\{C_nb_n\}_{n=0}^\infty$,  $\{C_ng_n\}_{n=0}^\infty$ and $\{C_na_n\}_{n=0}^\infty$ by
\begin{equation}\label{genfuncsABG}
\begin{aligned}
&\mathcal{B}^{U^3D}(t)=\sum_{n=0}^\infty C_nb_nt^n;\\
&\mathcal{G}^{U^3D}(t)=\sum_{n=0}^\infty C_ng_nt^n;\\
&\mathcal{A}^{U^3D}(t)=\sum_{n=0}^\infty C_na_nt^n.\\
\end{aligned}
\end{equation}

We will derive a system of three linear equations for $\mathcal{B}^{U^3D}(t)$, $\mathcal{G}^{U^3D}(t)$ and $\mathcal{A}^{U^3D}(t)$
and then solve for one of them explicitly.
From the definitions of $B^{U^3D}_n$, $G^{U^3D}_n$ and $A^{U^3D}_n$, we have
$$
B^{U^3D}_n\equiv0,\ 0\le n\le 4;\ \ G^{U^3D}_n\equiv0,\ 0\le n\le 1; \ \ A^{U^3D}_n=0,\ 0\le n\le 2 .
$$
Thus,
\begin{equation}\label{ABGbegin}
\begin{aligned}
&b_n=0,\ 0\le n\le 4;\\
&g_n=0,\ 0\le n\le 1;\\
&a_n=0,\ 0\le n\le 2.
\end{aligned}
\end{equation}
The following proposition is the key to obtaining a set of three linear equations for the generating functions $\mathcal{B}^{U^3D}(t)$, $\mathcal{G}^{U^3D}(t)$ and $\mathcal{A}^{U^3D}(t)$.
\begin{proposition}\label{conddist3}
\noindent i.
\begin{equation}\label{Bconddist3}
\begin{aligned}
&B^{U^3D}_n|\{\sigma_j=n\}\stackrel{\text{dist}}{=} A^{U^3D}_{j-1}+B^{U^3D}_{n-j},\ j\in[n-1],\ n\ge2;\\
&B^{U^3D}_n|\{\sigma_n=n\}\stackrel{\text{dist}}{=}B^{U^3D}_{n-1},\ n\ge 2,
\end{aligned}
\end{equation}
where
 on the right hand side of \eqref{Bconddist3}, $A^{U^3D}_{j-1}$ and $B^{U^3D}_{n-j}$ are understood to be independent.

\noindent ii.
\begin{equation}\label{Gconddist3}
\begin{aligned}
&G^{U^3D}_n|\{\sigma_1=n\}\stackrel{\text{dist}}{=}G^{U^3D}_{n-1}, \ n\ge2;\\
&G^{U^3D}_n|\{\sigma_j=n\}\stackrel{\text{dist}}{=} \left(A^{U^3D}_{j-1}+G^{U^3D}_{n-j}\right)1_{\{G^{U^3D}_{n-j}\neq0\}}+\left(\Bb_{j-1}+1\right)1_{\{G^{U^3D}_{n-j}=0\}},\\
& j\in\{2,\cdots, n\},\ n\ge2,
\end{aligned}
\end{equation}
where
 on the right hand side of \eqref{Gconddist3}, $A^{U^3D}_{j-1}$ and $G^{U^3D}_{n-j}$ are understood to be independent and   $B^{U^3D}_{j-1}$ and $G^{U^3D}_{n-j}$  are understood to be independent.

\noindent iii.
\begin{equation}\label{Aconddist3}
\begin{aligned}
&A^{U^3D}_n|\{\sigma_1=n\}\stackrel{\text{dist}}{=}A^{U^3D}_{n-1}, \ n\ge2;\\
&A^{U^3D}_n|\{\sigma_j=n\}\stackrel{\text{dist}}{=} \left(A^{U^3D}_{j-1}+A^{U^3D}_{n-j}\right)1_{\{A^{U^3D}_{n-j}\neq0\}}+G^{U^3D}_{j-1}1_{\{A^{U^3D}_{n-j}=0\}},\\
& j\in\{2,\cdots, n\},\ n\ge2,
\end{aligned}
\end{equation}
where
on the right hand side of \eqref{Aconddist3}, $A^{U^3D}_{j-1}$ and $A^{U^3D}_{n-j}$ are understood to be independent and $G^{U^3D}_{j-1}$  and $A^{U^3D}_{n-j}$   are understood to be independent.
\end{proposition}

\begin{proof}
The proof is similar to that of Proposition \ref{conddist}. The first line of \eqref{Bconddist3} and the second lines of \eqref{Gconddist3} and \eqref{Aconddist3} follow from the rather obvious equalities
$$
\begin{aligned}
&B^{U^3D}_n(\sigma)=B^{U^3D}_{n-1}(\sigma_1\cdots\sigma_{n-1}), \ \text{if}\ \sigma_n=n;\\
&G^{U^3D}_n(\sigma)=G^{U^3D}_{n-1}(\sigma_2\cdots\sigma_n), \ \text{if}\ \sigma_1=n;\\
&A^{U^3D}_n(\sigma)=A^{U^3D}_{n-1}(\sigma_2\cdots\sigma_n), \ \text{if}\ \sigma_1=n.
\end{aligned}
$$

Recall  the notation $\text{red}(\sigma)$ that was introduced in the paragraph containing \eqref{nprob}.
The first line of \eqref{Bconddist3} follows from the equality
\begin{equation}\label{Bcondsig3}
\begin{aligned}
&B^{U^3D}_n(\sigma)=A^{U^3D}_{j-1}(\text{red}(\sigma_1\cdots\sigma_{j-1}))+B^{U^3D}_{n-j}(\sigma_{j+1}\cdots\sigma_n), \ \text{if}\ \sigma_j=n,\\
& \text{for}\ j\in[n-1], n\ge2,
\end{aligned}
\end{equation}
along with the fact that under the conditioned  measure $P_n^{\text{av}(132)}|\{\sigma_j=n\}$, the permutation $\text{red}(\sigma_1\cdots\sigma_{j-1})\in S_{j-1}$ has the distribution
$P_{j-1}^{\text{av}(132)}$,  the permutation $\sigma_{j+1}\cdots\sigma_n\in S_{n-j}$ has the distribution $P_{n-j}^{\text{av}(132)}$, and these two permutations are independent.
The explanation for  \eqref{Bcondsig3}  is essentially the same as the explanation for \eqref{Bcondsig}.
Generically, $B^{U^3D}_n(\sigma)$ is the sum of two terms. One of the terms is  $A^{U^3D}_{j-1}\left(\text{red}(\sigma_1\cdots\sigma_{j-1})\right)$, which counts  the  number of full sets of $U^3D$ and then adds one for an extra  $UU$.
This extra $UU$, along with $\sigma_j=n$ and $\sigma_{j+1}$
supply an additional full set $U^3D$ which was counted by $A^{U^3D}_{j-1}\left(\text{red}(\sigma_1\cdots\sigma_{j-1})\right)$ (via the adding one for the extra $UU$).  The other term is  $B^{U^3D}_{n-j}(\sigma_{j+1}\cdots\sigma_n)$, which counts the remaining sets of $U^3D$.

The second line of \eqref{Gconddist3} follows from the equality
\begin{equation}\label{Gcondsig3}
\begin{aligned}
&G^{U^3D}_n(\sigma)= \left(A^{U^3D}_{j-1}(\text{red}(\sigma_1\cdots\sigma_{j-1}))+G^{U^3D}_{n-j}(\sigma_{j+1}\cdots\sigma_n)\right)1_{G^{U^3D}_{n-j}(\sigma_{j+1}\cdots\sigma_n)\neq0}+\\
&\left(\Bb_{j-1}\left(\text{red}(\sigma_1\cdots\sigma_{j-1})\right)+1\right)1_{G^{U^3D}_{n-j}(\sigma_{j+1}\cdots\sigma_n)=0},\\
&\ \text{if}\ \sigma_j=n,\ \text{for}\ j\in[n-1], n\ge2,
\end{aligned}
\end{equation}
along with the fact that under the conditioned  measure $P_n^{\text{av}(132)}|\{\sigma_j=n\}$, the permutation $\text{red}(\sigma_1\cdots\sigma_{j-1})\in S_{j-1}$ has the distribution
$P_{j-1}^{\text{av}(132)}$,  the permutation $\sigma_{j+1}\cdots\sigma_n\in S_{n-j}$ has the distribution $P_{n-j}^{\text{av}(132)}$, and these two permutations are independent.
The explanation for \eqref{Gcondsig3} in the case that $G^{U^3D}_{n-j}(\sigma_{j+1}\cdots\sigma_n)\neq0$ is  similar to the  reasoning for \eqref{Bcondsig3}.
We explain \eqref{Gcondsig3} in the case that $G^{U^3D}_{n-j}(\sigma_{j+1}\cdots\sigma_n)=0$ with an example, the same example used to explain
\eqref{Acondsig}
in the  case that $\Aa_{n-j}(\sigma_{j+1}\cdots\sigma_n)=0$. Consider $\sigma=435786921$.
So $n=9$ and $j=7$. We have $G^{U^3D}_{n-j}(\sigma_{j+1}\cdots\sigma_n)=G^{U^3D}_2(21)=0$.
We have $B^{U^3D}_{j-1}(\text{red}(\sigma_1\cdots\sigma_{j-1}))=B^{U^3D}_6(\text{red}(435786))=B^{U^3D}_6(213564)=1$ because the subsequence 23564 (as well as 13564) corresponds to $U^3D$. And
we have $G^{U^3D}_n(\sigma)=G^{U^3D}_9(435786921)=2$ because  the subsequence 457869 (as well as 357869)  corresponds to $U^3DU$.
(Note that $A^{U^3D}_{j-1}(\text{red}(\sigma_1\cdots\sigma_{j-1}))=A^{U^3D}_6(\text{red}(435786))=A^{U^3D}_6(213564)=1$ because the subsequence 235 (as well as  several others) corresponds to $UU$. Thus,
when $G^{U^3D}_{n-j}(\sigma_{j+1}\cdots\sigma_n)=0$,
it is not true in general that
$G^{U^3D}_n(\sigma)= A^{U^3D}_{j-1}(\text{red}(\sigma_1\cdots\sigma_{j-1}))+G^{U^3D}_{n-j}(\sigma_{j+1}\cdots\sigma_n)$.)

The second line of \eqref{Aconddist3} follows from the equality
\begin{equation}\label{Acondsig3}
\begin{aligned}
&A^{U^3D}_n(\sigma)= \left(A^{U^3D}_{j-1}(\text{red}(\sigma_1\cdots\sigma_{j-1}))+A^{U^3D}_{n-j}(\sigma_{j+1}\cdots\sigma_n)\right)1_{A^{U^3D}_{n-j}(\sigma_{j+1}\cdots\sigma_n)\neq0}+\\
&\left(G^{U^3D}_{j-1}\left(\text{red}(\sigma_1\cdots\sigma_{j-1})\right)\right)1_{A^{U^3D}_{n-j}(\sigma_{j+1}\cdots\sigma_n)=0},\\
&\ \text{if}\ \sigma_j=n,\ \text{for}\ j\in[n-1], n\ge2,
\end{aligned}
\end{equation}
along with the fact that under the conditioned  measure $P_n^{\text{av}(132)}|\{\sigma_j=n\}$, the permutation $\text{red}(\sigma_1\cdots\sigma_{j-1})\in S_{j-1}$ has the distribution
$P_{j-1}^{\text{av}(132)}$,  the permutation $\sigma_{j+1}\cdots\sigma_n\in S_{n-j}$ has the distribution $P_{n-j}^{\text{av}(132)}$, and these two permutations are independent.
The explanation for \eqref{Acondsig3} in the case that $A^{U^3D}_{n-j}(\sigma_{j+1}\cdots\sigma_n)\neq0$
is similar to the  reasoning for \eqref{Bcondsig3}.
We explain \eqref{Acondsig3} in the case that $A^{U^3D}_{n-j}(\sigma_{j+1}\cdots\sigma_n)=0$ with an example.
Consider $\sigma=786543921$. So $n=9$ and $j=7$. We have $A^{U^3D}_{n-j}(\sigma_{j+1}\cdots\sigma_n)=A^{U^3D}_2(21)=0$.
 We have $G^{U^3D}_{j-1}(\text{red}(\sigma_1\cdots\sigma_{j-1}))=G^{U^3D}_6(\text{red}(786543))=G^{U^3D}_6(564321)=1$ because the subsequence 56 corresponds to $U$.
And we have $A^{U^3D}_n(\sigma)=A^{U^3D}_9(786543921)=1$ because the subsequence 789 corresponds to $UU$.
(Note that $A^{U^3D}_{j-1}(\text{red}(\sigma_1\cdots\sigma_{j-1}))=A^{U^3D}_6(\text{red}(786543))=A^{U^3D}_6(564321)=0$.
Thus, the equality  $A^{U^3D}_n(\sigma)= A^{U^3D}_{j-1}(\text{red}(\sigma_1\cdots\sigma_{j-1}))+A^{U^3D}_{n-j}(\sigma_{j+1}\cdots\sigma_n)$ is not true in general when
 $A^{U^3D}_{n-j}(\sigma_{j+1}\cdots\sigma_n)=0$.)
\end{proof}

We now use Proposition \ref{conddist3} to derive a system of three linear equations for
$\mathcal{B}^{U^3D}(t)$, $\mathcal{G}^{U^3D}(t)$ and $\mathcal{A}^{U^3D}(t)$.
Note from \eqref{Bconddist} and \eqref{Bconddist3} that the conditional distributions of $\Bb(t)$ and $B^{U^3D}(t)$ are exactly the same except that
$\Aa(t)$  in \eqref{Bconddist} is replaced by   $A^{U^3D}(t)$ in \eqref{Bconddist3}. Thus, it follows from \eqref{finalB} that
\begin{equation}\label{finalB3}
\mathcal{B}^{U^3D}(t)=\frac{t\left(C(t)-1\right)\mathcal{A}^{U^3D}(t)}{1-t-tC(t)}.
\end{equation}

We now turn to $\mathcal{G}^{U^3D}(t)$.
Note that for $l\in\mathbb{N}$ and $\sigma\in S_l(132)$,  $G^{U^3D}_l(\sigma)=0$ only for  $\sigma=l\cdots 21$;
thus $P_l^{\text{av}(132)}(G^{U^3D}_l=0)=\frac1{C_l}$.
Using this with   \eqref{nprob} and \eqref{Gconddist3}, it follows that
\begin{equation}\label{gn}
\begin{aligned}
&g_n=E_n^{\text{av}(132)}G^{U^3D}_n=
\sum_{j=1}^nE_n^{\text{av}(132)}(G^{U^3D}_n|\sigma_j=n)P_n^{\text{av}(132)}(\sigma_j=n)=
\frac{C_0C_{n-1}}{C_n}g_{n-1}+\\
& \sum_{j=2}^n\left(a_{j-1}\left(1-\frac1{C_{n-j}}\right)+g_{n-j}\right)\frac{C_{j-1}C_{n-j}}{C_n}+
\sum_{j=2}^n\frac{b_{j-1}+1}{C_{n-j}}\thinspace\frac{C_{j-1}C_{n-j}}{C_n}.
\end{aligned}
\end{equation}
Multiplying both sides of  \eqref{gn} by $C_nt^n$, summing over $n$ from 2 to $\infty$ and using \eqref{ABGbegin}, we obtain
\begin{equation}\label{gequ}
\begin{aligned}
&\mathcal{G}^{U^3D}(t)=\sum_{n=2}^\infty C_ng_nt^n=\\
&t\sum_{n=2}^\infty C_{n-1}g_{n-1}t^{n-1}+t\sum_{n=2}^\infty\left(\sum_{j=2}^n a_{j-1}C_{j-1}C_{n-j}\right)t^{n-1}-
t\sum_{n=2}^\infty\big(\sum_{j=2}^n
a_{j-1}C_{j-1}\big)t^{n-1}+\\
&t\sum_{n=2}^\infty\left(\sum_{j=2}^nC_{j-1}g   _{n-j}C_{n-j}\right)t^{n-1}+t\sum_{n=2}^\infty\left(\sum_{j=2}^nb_{j-1}C_{j-1}\right)t^{n-1}+t\sum_{n=2}^\infty\left(\sum_{j=2}^n
C_{j-1}\right)t^{n-1}.
\end{aligned}
\end{equation}
Considerations almost identical to those that led from \eqref{aequ} to \eqref{prefinalA} yield
\begin{equation}\label{prefinalG}
\begin{aligned}
&\mathcal{G}^{U^3D}(t)(t)=t\Bigg(\mathcal{G}^{U^3D}(t)(t)+\mathcal{A}^{U^3D}(t)(t)C(t)-\frac{\mathcal{A}^{U^3D}(t)(t)}{1-t}+\mathcal{G}^{U^3D}(t)(t)\left(C(t)-1\right)+\\
&\frac{\mathcal{B}^{U^3D}(t)}{1-t}+\frac{C(t)-1}{1-t}\Bigg),
\end{aligned}
\end{equation}
which we write as
\begin{equation}\label{finalG}
\mathcal{G}^{U^3D}(t)(t)=\frac{\left(tC(t)-\frac t{1-t}\right)\mathcal{A}^{U^3D}(t)+\frac t{1-t}\mathcal{B}^{U^3D}(t)+\frac t{1-t}\left(C(t)-1\right)}{1-tC(t)}.
\end{equation}

We now turn to $\mathcal{A}^{U^3D}(t)$.
We need the following lemma.
\begin{lemma}\label{probA=0lemma}
\begin{equation}\label{probA=0}
P_l^{\text{av}(132)}(A^{U^3D}_l=0)=\frac{2^{l-1}}{C_l},\ l\in\mathbb{N}.
\end{equation}
\end{lemma}
\begin{proof}
For convenience, define
\begin{equation}\label{gammaprob}
\begin{aligned}
&\gamma_l=P_l^{\text{av}(132)}(A^{U^3D}_l=0),\ l\in\mathbb{N};\\
& \gamma_0=1.
\end{aligned}
\end{equation}
For $\sigma\in S_l$, distributed as  $P_l^{\text{av}(132)}$, and  conditioned on $\sigma_i=l$, the permutations $\text{red}(\sigma_1\cdots\sigma_{i-1})$ and $\sigma_{i+1}\cdots\sigma_l$ are independent and distributed
respectively as  $P_{i-1}^{\text{av}(132)}$ and $P_{l-i}^{\text{av}(132)}$. If $\sigma_i=l$,  then $A^{U^3D}_l(\sigma)=0$ if and only if
$G^{U^3D}_{i-1}(\text{red}(\sigma_1\cdots\sigma_{i-1}))=0$ and $A^{U^3D}_{l-i}(\sigma_{i+1}\cdots\sigma_l)=0$. Now
$G^{U^3D}_{i-1}(\text{red}(\sigma_1\cdots\sigma_{i-1}))=0$ if and only if
$\text{red}(\sigma_1\cdots\sigma_{i-1})=i-1\cdots 21$. Thus, $P_{i-1}^{\text{av}(132)}\left(G^{U^3D}_{i-1}(\text{red}(\sigma_1\cdots\sigma_{i-1}))=0\right)=\frac1{C_{i-1}}$.
Therefore, we have
$$
P_l^{\text{av}(132)}\left(A^{U^3D}_l(\sigma)=0|\sigma_i=l\right)=\frac{\gamma_{l-i}}{C_{i-1}}.
$$
Consequently,
\begin{equation*}
\gamma_l=P_l^{\text{av}(132)}(A^{U^3D}_l=0)=\sum_{i=1}^l\frac{C_{i-1}C_{l-i}}{C_l}\frac{\gamma_{l-i}}{C_{i-1}}=\sum_{i=1}^l\frac{C_{l-i}\gamma_{l-i}}{C_l},
\end{equation*}
which we write as
\begin{equation}\label{krecur}
k_l=\sum_{i=0}^{l-1}k_i,\ \ k_i=C_i\gamma_i.
\end{equation}
Multiply both sides of  \eqref{krecur} by $t^l$ and write the resulting equation as
\begin{equation}\label{krecuragain}
k_lt^l=t\sum_{i=0}^{l-1}k_it^it^{l-1-i}.
\end{equation}
Let $K(t)=\sum_{l=0}^\infty k_lt^l$. Summing \eqref{krecuragain}  over $l$ from 1 to $\infty$, one obtains after some algebra
$$
K(t)=1+\frac{tK(t)}{1-t},
$$
which yields
\begin{equation}\label{Kgen}
K(t)=\frac{1-t}{1-2t}=1+\frac t{1-2t}=1+\sum_{l=1}^\infty 2^{l-1}t^l.
\end{equation}
Thus, $C_l\gamma_l=k_l=2^{l-1},\ l\ge1$. Consequently
$P_l^{\text{av}(132)}(A^{U^3D}_l=0)=\gamma_l=\frac{2^{l-1}}{C_l}$.
\end{proof}

Using \eqref{gammaprob}  with   \eqref{nprob} and \eqref{Aconddist3}, it follows that
\begin{equation}\label{an3}
\begin{aligned}
&a_n=E_n^{\text{av}(132)}A^{U^3D}_n=\sum_{j=1}^nE_n^{\text{av}(132)}(A^{U^3D}_n|\sigma_j=n)P_n^{\text{av}(132)}(\sigma_j=n)=\\
& \frac{C_0C_{n-1}}{C_n}a_{n-1}+  \sum_{j=2}^n\left(a_{j-1}\left(1-\gamma_{n-j}\right)+a_{n-j}\right)\frac{C_{j-1}C_{n-j}}{C_n}+\\
&\sum_{j=2}^ng_{j-1}\gamma_{n-j}\thinspace\frac{C_{j-1}C_{n-j}}{C_n}.
\end{aligned}
\end{equation}
Multiplying both sides of \eqref{an3} by $C_nt^n$ and summing over $n$ from 2 to $\infty$, and recalling \eqref{ABGbegin}, we obtain
\begin{equation}\label{aequ3}
\begin{aligned}
&\mathcal{A}^{U^3D}(t)=\sum_{n=2}^\infty C_na_nt^n=t\sum_{n=2}^\infty C_{n-1}a_{n-1}t^{n-1}+t\sum_{n=2}^\infty\left(\sum_{j=2}^nC_{j-1}a_{j-1}C_{n-j}\right)t^{n-1}-\\
&t\sum_{n=2}^{\infty}\left(\sum_{j=2}^nC_{j-1}a_{j-1}\gamma_{n-j}C_{n-j}\right)t^{n-1}+t\sum_{n=2}^\infty\left(\sum_{j=2}^nC_{j-1}C_{n-j}a_{n-j}\right)t^{n-1}+\\
&t\sum_{n=2}^{\infty}\left(\sum_{j=2}^nC_{j-1}g_{j-1}\gamma_{n-j}C_{n-j}\right)t^{n-1}.
\end{aligned}
\end{equation}
By \eqref{krecur}, $\gamma_{n-j}C_{n-j}=k_{n-j}$. Using this with \eqref{Kgen}, we have
\begin{equation}\label{productgen}
\begin{aligned}
&\sum_{n=2}^{\infty}\left(\sum_{j=2}^nC_{j-1}a_{j-1}\gamma_{n-j}C_{n-j}\right)t^{n-1}=
\sum_{n=2}^{\infty}\left(\sum_{j=2}^nC_{j-1}a_{j-1}k_{n-j}\right)t^{n-1}=\\
&K(t)\mathcal{A}^{U^3D}(t)=\frac{1-t}{1-2t}\mathcal{A}^{U^3D}(t);\\
&\sum_{n=2}^{\infty}\left(\sum_{j=2}^nC_{j-1}g_{j-1}\gamma_{n-j}C_{n-j}\right)t^{n-1}=\sum_{n=2}^{\infty}\left(\sum_{j=2}^nC_{j-1}g_{j-1}k_{n-j}\right)t^{n-1}=\\
&K(t)\mathcal{G}^{U^3D}(t)=\frac{1-t}{1-2t}\mathcal{G}^{U^3D}(t).
\end{aligned}
\end{equation}
The two terms on the left hand sides of \eqref{productgen} appear on the right hand side of \eqref{aequ3}. The other terms on the right hand side of \eqref{aequ3} can be treated
via straightforward algebraic calculations,  similar to what was done in  previous calculations.
This allows for \eqref{aequ3} to be written term by term as
\begin{equation*}\label{prefinalA3}
\begin{aligned}
&\mathcal{A}^{U^3D}(t)=t\mathcal{A}^{U^3D}(t)+t\mathcal{A}^{U^3D}(t)C(t)-\frac{1-t}{1-2t}\mathcal{A}^{U^3D}(t)+\\
&t\mathcal{A}^{U^3D}(t)\left(C(t)-1\right)
+\frac{1-t}{1-2t}\mathcal{G}^{U^3D}(t),
\end{aligned}
\end{equation*}
which yields
\begin{equation}\label{finalA3}
\mathcal{A}^{U^3D}(t)=\frac{\frac{t(1-t)}{1-2t}\mathcal{G}^{U^3D}(t)}{1-2tC(t)+\frac{t(1-t)}{1-2t}}.
\end{equation}

Now \eqref{finalB3}, \eqref{finalG} and \eqref{finalA3} provide a system of three linear equations for the three generating functions
$\mathcal{B}^{U^3D}(t), \mathcal{G}^{U^3D}(t)$ and $\mathcal{A}^{U^3D}(t)$.
Since
$G^{U^3D}_n(\sigma), A^{U^3D}_n(\sigma)\in\{B^{U^3D}_n(\sigma)-1,B^{U^3D}_n(\sigma),B^{U^3D}_n(\sigma)+1\}$, for all $n\in\mathbb{N}$ and all $\sigma\in S_n$,
 the leading order asymptotic behavior is  the same for
$E_n^{\text{av}(132)}B^{U^3D}_n, E_n^{\text{av}(132)}G^{U^3D}_n$ and $E_n^{\text{av}(132)}A^{U^3D}_n$.
Thus, it doesn't matter which of the generating functions we solve for.
  We will solve for
$\mathcal{G}^{U^3D}(t)$.
We start with \eqref{finalG}, and replace the term $\mathcal{B}^{U^3D}(t)$
on the right hand side of \eqref{finalG} with the right hand side of \eqref{finalB3}. After rearranging some terms, this gives
\begin{equation}\label{GAremain}
\begin{aligned}
&\mathcal{G}^{U^3D}(t)=\frac{t(C(t)-1)}{(1-t)\left(1-tC(t)\right)}+\\
&\left(tC(t)-\frac t{1-t}+\frac{t^2(C(t)-1)}{(1-t)\left(1-t-tC(t)\right)}\right)\frac1{1-tC(t)}\mathcal{A}^{U^3D}(t).
\end{aligned}
\end{equation}
Now we replace $\mathcal{A}^{U^3D}(t)$ on the right hand side of \eqref{GAremain} with the right hand side of \eqref{finalA3}.
This yields an equation in which only the generating function  $\mathcal{G}^{U^3D}(t)$ appears. Solving for $\mathcal{G}^{U^3D}(t)$, we obtain
\begin{equation}\label{Galone}
\begin{aligned}
&\mathcal{G}^{U^3D}(t)=\frac{t(C(t)-1)}{(1-t)\left(1-tC(t)\right)(1-d_1(t))},\ \text{where}\\
&d_1(t)=\left(tC(t)-\frac t{1-t}+\frac{t^2(C(t)-1)}{(1-t)(1-t-tC(t))}\right)
\left(\frac1{1-tC(t)}\right)
\left(\frac{t(1-t)}{1-2t}\right)\times\\
&\left(\frac1{1-2tC(t)+\frac{t(1-t)}{1-2t}}\right).
\end{aligned}
\end{equation}

\section{ Completion of the proof of   Theorem \ref{thmUUUD}}\label{sec4}

In \eqref{Galone},
when we perform the multiplication $(1-t)\left(1-tC(t)\right)d_1(t)$,
 the second of the four factors in $d_1(t)$ will disappear, and the $1-t$ in the denominator of two of the terms in the first factor will also disappear. We obtain
 \begin{equation}\label{mult}
 \begin{aligned}
& (1-t)\left(1-tC(t)\right)d_1(t)=\left((1-t)tC(t)-t+\frac{t^2(C(t)-1)}{1-t-tC(t)}\right)\left(\frac{t(1-t)}{1-2t}\right)\times\\
&\left(\frac1{1-2tC(t)+\frac{t(1-t)}{1-2t}}\right).
 \end{aligned}
 \end{equation}
 Multiplying  the denominators of the   second and third factors on the right hand side of \eqref{mult}, we have
 \begin{equation}\label{mult2}
 (1-2t)\left(1-2tC(t)+\frac{t(1-t)}{1-2t}\right)=1-t-t^2-2tC(t)+4t^2C(t).
 \end{equation}
 Thus, multiplying both the numerator and the denominator on
 the right hand side of \eqref{Galone} by $1-t-t^2-2tC(t)+4t^2C(t)$, and using \eqref{mult} and \eqref{mult2}, we obtain
\begin{equation}\label{Galone2}
\mathcal{G}^{U^3D}(t)=\frac{t(C(t)-1)\left(1-t-t^2-2tC(t)+4t^2C(t)\right)}{d_2(t)},
\end{equation}
where
\begin{equation}\label{d2}
\begin{aligned}
&d_2(t)=(1-t)(1-tC(t))\left(1-t-t^2-2tC(t)+4t^2C(t)\right)-\\
&\left((1-t)tC(t)-t+\frac{t^2(C(t)-1)}{1-t-tC(t)}\right)t(1-t).
\end{aligned}
\end{equation}
Multiplying the numerator and the denominator  on the right hand side of \eqref{Galone2} by $1-t-tC(t)$, and using \eqref{d2}, we obtain
\begin{equation}\label{Galone3}
\mathcal{G}^{U^3D}(t)=\frac{n(t)}{d(t)},
\end{equation}
where
\begin{equation}\label{nd}
\begin{aligned}
&n(t)=t(C(t)-1)\left(1-t-t^2-2tC(t)+4t^2C(t)\right)\left(1-t-tC(t)\right);\\
&d(t)=(1-t)(1-tC(t))\left(1-t-t^2-2tC(t)+4t^2C(t)\right)\left(1-t-tC(t)\right)-\\
&\Big(\left((1-t)tC(t)-t\right)\left(1-t-tC(t)\right)+t^2(C(t)-1)\Big)t(1-t).
\end{aligned}
\end{equation}

Grouping powers of $C(t)$, we have
\begin{equation}\label{nd2nd}
\begin{aligned}
&n(t)=A_3(t)C^3(t)+A_2(t)C^2(t)+A_1(t)C(t)+A_0(t);\\
&d(t)=B_3(t)C^3(t)+B_2(t)C^2(t)+B_1(t)C(t)+B_0(t),
\end{aligned}
\end{equation}
where
\begin{equation}\label{ABcoef}
\begin{aligned}
&A_3(t)=2t^3-4t^2;\ \ A_2(t)=t^4+5t^3-3t^2;\\
& A_1(t)=4t^4-7t^3+t^2+t;\ \ A_0(t)=-t^4+2t^2-t;\\
&B_3(t)=-4t^5+6t^4-2t^3;\ \ B_2(t)=-2t^5+12t^4-15t^3+5t^2;\\
&B_1(t)=2t^5+t^4-11t^3+12t^2-4t;\ \ B_0(t)=-t^4+3t^2-3t+1.
\end{aligned}
\end{equation}
Recalling the formula for $C(t)$ in \eqref{Catgenfunc}, we have
\begin{equation}\label{C23}
\begin{aligned}
&C^2(t)=\frac{1-2t-\sqrt{1-4t}}{2t^2};\\
&C^3(t)=\frac{1-3t-(1-t)\sqrt{1-4t}}{2t^3}.
\end{aligned}
\end{equation}
Letting
$$
R:=\sqrt{1-4t}
$$
and substituting from \eqref{Catgenfunc} and \eqref{C23} in \eqref{nd2nd}, we obtain after a lot of algebra
\begin{equation}\label{nd3rd}
\begin{aligned}
&n(t)=\left(-\frac{(1-t)A_3(t)}{2t^3}-\frac{A_2(t)}{2t^2}-\frac{A_1(t)}{2t}\right)R+\\
&\frac{(1-3t)A_3(t)}{2t^3}+\frac{(1-2t)A_2(t)}{2t^2}+\frac{A_1(t)}{2t}+A_0(t);\\
&d(t)=\left(-\frac{(1-t)B_3(t)}{2t^3}-\frac{B_2(t)}{2t^2}-\frac{B_1(t)}{2t}\right)R+\\
&\frac{(1-3t)B_3(t)}{2t^3}+\frac{(1-2t)B_2(t)}{2t^2}+\frac{B_1(t)}{2t}+B_0(t).
\end{aligned}
\end{equation}
Using \eqref{ABcoef}, one finds that
\begin{equation}\label{AsBs}
\begin{aligned}
&-\frac{(1-t)A_3(t)}{2t^3}-\frac{A_2(t)}{2t^2}-\frac{A_1(t)}{2t}=t^2(1-2t);\\
&\frac{(1-3t)A_3(t)}{2t^3}+\frac{(1-2t)A_2(t)}{2t^2}+\frac{A_1(t)}{2t}+A_0(t)=t^3(1-t);\\
&-\frac{(1-t)B_3(t)}{2t^3}-\frac{B_2(t)}{2t^2}-\frac{B_1(t)}{2t}=-t^4-\frac32t^3+\frac92t^2-\frac52t+\frac12;\\
&\frac{(1-3t)B_3(t)}{2t^3}+\frac{(1-2t)B_2(t)}{2t^2}+\frac{B_1(t)}{2t}+B_0(t)=2t^4-\frac{13}2t^3+\frac{15}2t^2-\frac72t+\frac12.
\end{aligned}
\end{equation}
From \eqref{nd3rd} and  \eqref{AsBs}, we have
\begin{equation}\label{nd4th}
\begin{aligned}
&n(t)=t^2(1-2t)R+t^3(1-t);\\
&d(t)=\left(-t^4-\frac32t^3+\frac92t^2-\frac52t+\frac12\right)R+2t^4-\frac{13}2t^3+\frac{15}2t^2-\frac72t+\frac12.
\end{aligned}
\end{equation}

Recall that $R=\sqrt{1-4t}$. In order to eliminate the square root in the denominator $d(t)$ in \eqref{nd4th}, we multiply the numerator and denominator by the denominator's conjugate,
$-\left(-t^4-\frac32t^3+\frac92t^2-\frac52t+\frac12\right)R+2t^4-\frac{13}2t^3+\frac{15}2t^2-\frac72t+\frac12$.
Calling the resulting numerator and denominator by $\bar n(t)$ and $\bar d(t)$, 
this yields
\begin{equation}\label{nd5th}
\begin{aligned}
&\bar n(t)=\left(-t^6-\frac92t^5+21t^4-\frac{57}2t^3+\frac{35}2t^2-5t+\frac12\right)\sqrt{1-4t}+\\
&6t^6+\frac{29}2t^5-58t^4+\frac{119}2t^3-\frac{55}2t^2+6t-\frac12;\\
&\bar d(t)=4t^7+15t^6-56t^5+45t^4+4t^3-19t^2+8t-1.
\end{aligned}
\end{equation}
The new denominator factors  as
\begin{equation}\label{dfactor}
\begin{aligned}
&\bar d(t)=4t^7+15t^6-56t^5+45t^4+4t^3-19t^2+8t-1=\\
&(1-4t)(1-t)^2\left(-t^4-6t^3+2t^2+2t-1\right).
\end{aligned}
\end{equation}
The two polynomials in the new numerator $\bar n(t)$ factor as
\begin{equation}\label{nfactor}
\begin{aligned}
&-t^6-\frac92t^5+21t^4-\frac{57}2t^3+\frac{35}2t^2-5t+\frac12=\\
&(1-t)^2\left(-t^4-\frac{13}2t^3+9t^2-4t+\frac12\right);\\
&6t^6+\frac{29}2t^5-58t^4+\frac{119}2t^3-\frac{55}2t^2+6t-\frac12=\\
&(1-4t)(1-t)\left(\frac32t^4+\frac{11}2t^3-8t^2+\frac72t-\frac12\right).
\end{aligned}
\end{equation}
From \eqref{nd5th}-\eqref{nfactor} and \eqref{Galone3}, we conclude that
\begin{equation}\label{Gfinal}
\begin{aligned}
&\mathcal{G}^{U^3D}(t)=\frac{-t^4-\frac{13}2t^3+9t^2-4t+\frac12}{-t^4-6t^3+2t^2+2t-1}\left(1-4t\right)^{-\frac12}+\\
&\frac{\frac32t^4+\frac{11}2t^3-8t^2+\frac72t-\frac12}{(1-t)\left(-t^4-6t^3+2t^2+2t-1\right)}.
\end{aligned}
\end{equation}

The smallest absolute value among the  roots of $-t^4-6t^3+2t^2+2t-1$
is larger than $\frac14$; thus, $\frac1{-t^4-6t^3+2t^2+2t-1}$ and $\frac1{(1-t)\left(-t^4-6t^3+2t^2+2t-1\right)}$ are analytic
in a ball centered at the origin of radius larger than $\frac14$. Thus, applying \eqref{transferpower} in the case $\alpha=\frac12$ with
$g(t)=\frac1{-t^4-6t^3+2t^2+2t-1}$ and with $g(t)=\frac1{(1-t)\left(-t^4-6t^3+2t^2+2t-1\right)}$,
it follows from \eqref{Gfinal} that the leading order asymptotic contribution to $[t^n]\mathcal{G}^{U^3D}(t)$ comes from the term
$\frac{-t^4-\frac{13}2t^3+9t^2-4t+\frac12}{-t^4-6t^3+2t^2+2t-1}\left(1-4t\right)^{-\frac12}$.
Since $\frac1{-t^4-6t^3+2t^2+2t-1}|_{t=\frac14}=-\frac{256}{121}$, we conclude from \eqref{transferpower} that
\begin{equation}\label{finalasympG}
\begin{aligned}
&[t^n]\mathcal{G}^{U^3D}(t)\sim\frac{256}{121}4^n\frac{n^{-\frac12}}{\sqrt\pi}\left(4^{-4}+\frac{13}2\cdot4^{-3}-9\cdot4^{-2}+4\cdot4^{-1}-\frac12\right)=\\
&\frac{256}{121}4^n\frac{n^{-\frac12}}{\sqrt\pi}\frac{11}{256}=\frac1{11}4^n\frac{n^{-\frac12}}{\sqrt\pi}.
\end{aligned}
\end{equation}
From \eqref{genfuncsABG} and \eqref{anbngn}, we have
$[t^n]\mathcal{G}^{U^3D}(t)=C_ng_n=C_nE_n^{\text{av}(132)}G^{U^3D}_n$.
As previously noted,
the Catalan numbers satisfy
$C_n\sim4^n\frac{n^{-\frac32}}{\sqrt\pi}$.
Using these facts with \eqref{finalasympG}, we conclude that
\begin{equation}\label{1-11}
E_n^{\text{av}(132)}G^{U^3D}_n\sim\frac1{11}n.
\end{equation}
Theorem \ref{thmUUUD} now follows from \eqref{1-11}, \eqref{BnLn3} and the fact that the leading order asymptotic behavior of
$E_n^{\text{av}(132)}B^{U^3D}_n$ and of $E_n^{\text{av}(132)}G^{U^3D}_n$ coincide.
\hfill$\square$

\end{document}